\documentclass[12pt]{article}

\usepackage{amsmath,amsthm,amssymb}

\def\Prob{\operatorname{Prob}}

\def\ex{\operatorname{ex}}

\theoremstyle{plain}
\newtheorem{theorem}{Theorem}

\newtheorem{definition}{Definition}
\newtheorem{corollary}{Corollary}

\newtheorem{statement}{Proposition}

\theoremstyle{definition}

\newtheorem{remark}{Remark}

\begin{document}

\title{Equipped graded graphs, projective limits of simplices,
and their boundaries}
 \author{A.~M.~Vershik\thanks{St.~Petersburg Department of Steklov Institute of Mathematics, St.~Petersburg, Russia;
Institute for Information Transmission Problems, Moscow, Russia. Email: {\tt vershik@pdmi.ras.ru}. Supported by the RNF grant 14-11-00581.}}

\date{January 17, 2015.}


  \maketitle

\begin{abstract}
In this paper, we develop a theory of equipped graded graphs (or Bratteli
diagrams) and an alternative theory of projective limits of  finite-dimensional
simplices. An equipment is an additional structure on the
graph, namely, a system of ``cotransition'' probabilities on the set of its
paths. The main problem is to describe all probability measures on the
path space of a graph with given cotransition probabilities; it goes back
to the problem, posed by E.~B.~Dynkin in the 1960s, of describing exit and
entrance boundaries for Markov chains. The most important example is
the problem of describing all central measures, to which  one can reduce the problems of describing states
on AF-algebras or characters on locally finite groups. We suggest an unification of the whole theory, an interpretation of the
notions of Martin, Choquet, and Dynkin boundaries in terms of equipped
graded graphs and in terms of the theory of projective limits of simplices.
In the last section, we study the new notion of ``standardness'' of projective limits of simplices and of equipped Bratteli diagrams, as well as the
notion of ``lacunarization.''
\end{abstract}

\section{Introduction}
In functional analysis, probability theory, geometry one considers different notions of boundary. That of Martin boundary is popular in potential theory, the theory of harmonic functions, the theory of Markov processes. This notion, which appeared in the 1940s, gave rise to an extensive literature. It is the notion of Martin boundary that was one of the motivations for developing the Choquet theory (the other one being the famous Krein--Milman theorem about convex compact subsets in a locally convex space). Here we want to establish a direct connection between the two theories: the theory of boundaries of Markov processes, which is in fact the theory of invariant measures on path spaces of graded graphs (or, which is the same, on the spaces of trajectories of Markov compacta) and the geometry of projective limits of simplices. This simple connection helps in solving the problem of finding boundaries, which, in turn, includes the general theory of invariant measures for hyperfinite actions of groups and hyperfinite equivalence relations. From an algebraic point of view, these  problems concern  lists of traces on algebras or characters on groups. In fact, already in the first papers by the author and S.~V.~Kerov, D.~Voiculescu, and others, a connection was established between the theory of central measures on path spaces of graded graphs and the theory of traces on locally semisimple algebras.  The main nontrivial example was, of course, that of the Young graph, as well as the Kingman graph, etc. Then other examples were considered. The Martin boundary was studied in
\cite{KOO,K} in connection with several concrete graphs. Here we want to describe a general interpretation and a more direct connection of different boundaries with the convex geometry of projective limits of simplices. We relate the exit boundary, which was introduced in the framework of Markov theory in remarkable papers by E.~B.~Dynkin in the 1960s, directly with the Choquet boundary. Other boundaries also find natural interpretations in the framework of the geometry of simplices. But the main results of this  paper, in continuation of the recent paper
 \cite{V}, are presented in \S4; there we introduce new general notions, those of the intrinsic metric (topology), standardness, etc. We use and apply ideas and results of the theory of metric filtrations (\cite{Dec}). The final goal is to find the boundaries in the (numerous) cases where this can be done (standard graphs) and to recognize the cases where the problem is unsolvable (totally nonstandard graphs). One of the obvious applications is to the essentially new problem from the theory of random walks on groups, that of describing all conditional walks with uniform cotransition probabilities. In a joint paper with A.~V.~Malyutin, which is currently in preparation, we will consider one of such examples.

\section{The input data of the problem}

\subsection{A graded graph, the path space, topology}
Consider a locally finite, infinite $\mathbb N$-graded graph~$\Gamma$ (a Bratteli diagram).
The set of vertices of degree $n$, $n=0,1, \dots$, will be denoted by $\Gamma_n$ and called the $n$th {\it level} of~$\Gamma$:
 $$\Gamma=\coprod_{n\in \mathbb N} \Gamma_n;$$
the level $\Gamma_0$ consists of the single initial vertex  $\{\emptyset\}$. We assume that every vertex has at least one successor, and every vertex except the initial one has at least one predecessor. In what follows, we also assume that the edges of $\Gamma$ are simple.\footnote{For our purposes, allowing Bratteli diagrams to have multiple edges does not give anything new, since cotransition probabilities introduced below replace and generalize multiplicities of edges.}
No other assumptions are imposed. A locally semisimple algebra
 $A(\Gamma)$ over $\mathbb C$ is canonically associated to a graded graph $\Gamma$; however, here we do not consider this algebra and do not discuss the relation of the notions  introduced below with this algebra and its representations; this problem is worth a separate study.

A path in $\Gamma$ is a (finite or infinite) sequence of vertices of $\Gamma$ in which every pair of neighboring vertices is connected by an edge (for graphs without multiple edges, this is the same as a sequence of edges). The space of all infinite paths in  $\Gamma$ is denoted by $T(\Gamma)$; it is, in a natural sense, the inverse limit of the spaces of finite paths (leading from the initial vertex to vertices of some fixed level), and thus is a Cantor-like compact set. Cylinder sets in $T(\Gamma)$ are sets defined in terms of conditions on initial segments of paths up to level $n$; they are clopen and determine a base of the topology of $T(\Gamma)$. There is a natural notion of {\it tail equivalence relation} $\tau_{\Gamma}$ on $T(\Gamma)$: two infinite paths are tail-equivalent if they eventually coincide; one also says that such  paths lie in the same block of
the {\it tail partition}.
  The {\it tail filtration} $\Xi(\Gamma)=\{{\mathfrak A}_0 \supset {\mathfrak A}_1 \supset \dots\}$
  is the decreasing sequence of $\sigma$-algebras ${\mathfrak A}_n$, $n \in \mathbb N$,  where ${\mathfrak A}_n$  consists of all Borel sets $A\subset T(\Gamma)$ such that along with every path
$A$ contains all paths coinciding with it up to the $n$th level. In an obvious sense,  ${\mathfrak A}_n$  is complementary to the finite $\sigma$-algebra of cylinder sets of order
$n$. The key idea of \cite{V,V1} is to apply the theory of decreasing filtrations (see, e.g.,
\cite{Dec}) to the analysis of the structure of path spaces and measures on them. Below we touch on this problem.

  \subsection{A system of cotransition matrices, an equipped graph}

 Given a graded graph $\Gamma$, we introduce an additional structure on this graph,  namely, a {\it system of cotransition probabilities}
 $$\Lambda=\{\lambda=\lambda_v^u;\; u\in \Gamma_n, v\in \Gamma_{n+1}, (u,v)\in \mbox{edge}(\Gamma_n,\Gamma_{n+1}),\;n=0,1, \dots\},$$
 by associating with each vertex $v \in \Gamma_n$ a probability vector whose component
  $\lambda_v^u$ is the probability
 of an edge $u\prec v$ entering $v$ from the previous level; here $\sum\limits_ {u: u\prec v}\lambda_v^u=1$ and $\lambda_v^u\geq 0$.

\begin{definition}
An equipped graph is a pair $(\Gamma, \Lambda)$ where $\Gamma$ is a graded graph and
$\Lambda$ is a system of cotransition probabilities on its edges.
 \end{definition}

The term ``cotransition probabilities'' is borrowed from the theory of Markov chains (\cite{D}):
if we regard the vertices of $\Gamma$ as the states of a Markov chain starting from the state $\varnothing$ at time $t=0$,
and the numbers of levels as moments of time, then  $\Lambda=\{\lambda_v^u\}$ is interpreted as the system of cotransition probabilities for this Markov chain:
 $$\mbox{Prob} \{x_{t}=u|x_{t+1}=v\}=\lambda_v^u. $$

It is convenient to regard the  system of cotransition probabilities as a system of $d_n\times d_{n+1}$ Markov matrices:
 $$
 \{\lambda_v^u\}, \quad u \in \Gamma_n, v \in \Gamma_{n+1};\quad |\Gamma_n|=d_n,\; |\Gamma_{n+1}|=d_{n+1}, \;n \in \mathbb N;
 $$
 these matrices generalize the  ($0\vee 1$) incidence matrices of the graph
 $\Gamma$. Our main interest lies in the asymptotic properties of this sequence of matrices. In this sense, the whole theory developed here is a part of the asymptotic theory of infinite products of Markov matrices, which is important in itself.

Every Borel measure $\mu$ on the path space of a graph determines a  system of cotransition probabilities as the system of conditional measures of natural measurable partitions. One says that  $\mu$ agrees with a given system $\Lambda$
of cotransition probabilities if the collection of cotransition probabilities of $\mu$ (for all vertices) coincides with $\Lambda$. Recall that in general  a system of cotransition probabilities does not allow one to uniquely determine the system of transition probabilities $$
 \mbox{Prob} \{x_{t+1}=v|x_t=u\};
 $$
 in other words, it does not uniquely determine the Markov chain.

A measure on the path space of a graph is called {\it ergodic} if the tail $\sigma$-algebra (i.e., the intersection of all  $\sigma$-algebras of the tail filtration) is trivial $\bmod 0$, i.e., consists of two elements.

{\it Our aim is to enumerate all Markov measures, i.e., all possible transition probabilities of ergodic Markov chains with a given system of cotransition probabilities $\Lambda$. It is natural to call the corresponding list the Dynkin boundary of the system of cotransition probabilities $\Lambda$, or the Dynkin boundary of the equipped graph $(\Gamma,\Lambda)$}. This is a topological boundary, and, as we will see, it is the Choquet boundary of a certain simplex (a projective limit of finite-dimensional simplices).

In the probability literature (e.g., in the theory of random walks),  cotransition probabilities are usually defined not explicitly, but as the cotransition probabilities of a given Markov process. We prefer to define them directly, i.e., include them into the input data of the problem. And if such a Markov chain, i.e., a Markov measure
$\mu$ on the path space with these cotransition probabilities is already defined, we can consider the {\it metric Poisson--Furstenberg boundary of this measure}, a space with a measure defined on the tail
 $\sigma$-algebra and induced by $\mu$. This boundary, regarded as a measure space, is a part of the Dynkin boundary. The system of cotransition probabilities determines a cocycle on the tail equivalence relation, i.e., a function $(\gamma_1,\gamma_2) \rightarrow c(\gamma_1,\gamma_2)$ of a pair of equivalent paths, which  is equal to the ratio of the products of cotransition probabilities along these paths (such a ratio is finite,
 since the paths are equivalent). In statistical physics and the theory of configurations, one  also considers more general cocycles called Radon--Nikodym cocycles. In our case, the cocycle has a special form (the product of probabilities over edges) and is called a Markov cocycle.  {\it A measure with given cotransition probabilities is a measure with a given Radon--Nikodym cocycle for a transformation group whose orbit partition coincides with the tail partition.}\footnote{The path space $T(\Gamma)$ of a graded graph $\Gamma$ has another additional structure, namely, a linear order on the edges entering each vertex. It allows one to introduce a lexicographic order on each class of tail-equivalent paths, and then define the so-called adic transformation on $T(\Gamma)$, which sends a path to the next path in the sense of this order. In this paper, we do not use this structure, restricting ourselves to the following remark: a measure is central if and only if it is invariant under the adic transformation, and, as observed above,  a measure  has a given system of cotransition probabilities   if and only if the adic transformation has a given Radon--Nikodym cocycle.}

The most important special case of a system of cotransition probabilities, which is studied in combinatorics, representation theory, and algebraic settings, is as follows:
 $$
 \lambda_v^u=\frac{\dim(u)}{\sum\limits_ {u:u\prec v} \dim(u)},
 $$
where $\dim(u)$ is the number of paths leading from the initial vertex $\emptyset$ to  $u$ (i.e., the dimension of the representation of the algebra $A(\Gamma)$ corresponding to the vertex $u$). In other words, the probability to get from $v$ to $u$ is equal to the fraction of paths that lead from   $\emptyset$ to $u$ among all the paths that lead from
 $\emptyset$ to  $v$. This system of cotransition probabilities is canonical, in that it is determined only by the graph. The corresponding Markov measures on the path space $T(\Gamma)$ are called {\it central measures}; up to now, they have been studied only in the literature on Bratteli diagrams. In terms of the theory of
C$^*$-algebras, central measures are traces on the algebra $A(\Gamma)$, and ergodic central measures are indecomposable traces. For more details on the case of central measures, see \cite{V} and the extensive literature of
the 1980s--2000s. However, further development of the whole theory requires considering an arbitrary system of cotransition probabilities. Note that already for central measures, the asymptotic behavior can be very different; the example of the graph of unordered pairs from   \cite{V} shows how much the answer can differ from the case of familiar graphs, such as the Young graph.

An analog of a system of cotransition probabilities, i.e., the notion of an equipped graph, can also be defined in greater generality: instead of a graded graph, it suffices to have a directed graph or multigraph whose each vertex (except possibly one) has a nonempty set of ingoing edges; one can define an arbitrary system of probabilities on the set of ingoing edges of every vertex; the problem is still to describe the Dynkin boundary, i.e., the collection of all measures on the set of directed paths  with given conditional entrance probabilities.

 \subsection{Operators, the Martin boundary, terminology}

An equipped graph, or a pair $(\Gamma, \Lambda)$, gives rise naturally to two linear operators. The first operator
 $L=L(\Gamma, \Lambda)$ acts on the space $F(\Gamma)=\{f:\Gamma \rightarrow \mathbb R\}$ of all real functions on the set of vertices of  $\Gamma$:
$$(Lf)(v)=\sum\limits_ {u:u\prec v} \lambda_v^u f(u);$$
it is the operator of averaging a function over the cotransition probabilities. The other operator
 $L^*$ acts on the simplex $\Sigma(\Gamma)$ of all probability measures on the set of vertices of $\Gamma$:
    $$L^*(\mu)(u)=\sum\limits_ {v:v\succ u} \lambda_v^u \mu(v).$$
The first operator sends functions on the $n$th level of the graph (i.e., functions vanishing 
at all levels except $n$)  to functions on the $(n+1)$th level, and the second one sends probability measures on the $(n+1)$th level to probability measures on the $n$th level. It is these restrictions of operators that are determined by the transition and cotransition matrices mentioned above.

It is also clear that every vertex $v$ of the $(n+1)$th level (more exactly, the corresponding $\delta$-measure $\delta_v$)
correctly defines a unique probability measure  $L^*(\delta_v)=\lambda_v^{\cdot}$ on the $n$th level and
(by induction, repeatedly applying the operator $L^*$) measures on all previous levels
$\Gamma_k$, $k<n+1$. We will denote these measures by $\mu_v^k$, $v\in \Gamma_{n+1}$, $k<n+1$.

Following the tradition, by the Martin kernel we mean the function $K(u,v)$ of pairs of vertices $u,v$
from different levels that can be connected by at least one path (otherwise $K(u,v)=0$):
 $$K(u,v)=\sum\limits_ {\{w_i\}}\prod_i p(w_i,w_{i+1});\quad u=w_0\prec w_1 \prec \dots \prec w_k=v, $$
where the sum is over all paths leading from $u\in \Gamma_n$ to $v\in \Gamma_{n+1}.$

As usual, the Martin kernel allows one to define the Martin compactification $\tilde M(\Gamma, \Lambda)$ of the set of functions $\Gamma'=\{v\mapsto K(u,v); u\in \Gamma\}\subset F(\Gamma)$ with respect to the pointwise convergence.
The difference
$\tilde M(\Gamma,\Lambda) \smallsetminus \Gamma'$ is called the {\it Martin boundary} of the pair $(\Gamma, \Lambda)$.
In more detail, this means the following: a sequence of vertices $v_k\in \Gamma$, $k=1,2, \dots$, is a Cauchy sequence if for every vertex $u\in \Gamma$ the numerical sequence of the probabilities $\Prob\{u|v_k\}$ is a Cauchy sequence as $k \to \infty$. More simply stated, the sequence of the measures $\mu_{v_k}^n$ induced by  $v_k$ at an arbitrary level $n$ of $\Gamma$ weakly converges to a measure $\mu^n$ determined by the sequence $\{v_k\}$.
We identify sequences that determine the same measure. Thus a point of the Martin boundary is a probability measure on the space of infinite paths in $\Gamma$ that has given cotransition probabilities and is the weak limit of an infinite sequence of measures $\mu_n$, $n=1,2, \dots$. This is the conventional definition of the Martin boundary of a Markov chain, but stated in terms related to a graded graph. Below we will give an equivalent abstract definition of the Martin boundary in terms of projective limits of simplices and return to the question posed above.

\begin{remark}{\rm In the context of Markov chains and processes,
Dynkin (see one of the first papers \cite{D})  uses the terms ``exit boundary'' and ``entrance boundary.'' What we have called above the Dynkin boundary can be interpreted as the exit boundary for the system of cotransition probabilities under consideration, since it refers to the final behavior of trajectories of Markov chains; and a point of the boundary is a class of trajectories that have the same final behavior and thus determine the same (conditional) ergodic Markov chain.}
\end{remark}

Unfortunately,  confusion in terminology is likely to arise here: the operator conjugate to that of averaging functions over the cotransition probabilities (the ``Laplace operator'') determines a boundary, which, not without reason, is usually called the entrance boundary. The papers \cite{K,KOO} use exactly this terminology. The thing is,  what we consider to be primary, the generator (Laplace operator) and Green's function acting in the space of functions, as in potential theory, or the projection operator on measures, which determines the geometry of the Markov chain in a more general context. On the other hand, of course, time inversion turns the problem of describing all ergodic measures with given cotransition probabilities into the problem of describing all measures with given transition probabilities, which, according to our definition, should be called the entrance boundary. We try to avoid this confusion and  do not use these terms. Our new definitions -- those of standardness, compactness, lacunarity, intrinsic metric and topology -- apply both to systems of transition and cotransition probabilities, and to the theory of ordinary and reverse martingales.

Another remark concerns the notion of the entropy of a (nonstationary) Markov chain. There should exist general theorems that would relate the normalized entropy of a Markov chain on the path space of a graph (cf.\ the definition of the entropy of a Young diagram in \cite{VeK}) with the entropy of random walks (see, e.g., \cite{KV}).

Finally, note another very important fact. The theory of boundaries of Markov processes we develop should be constructed separately for $\sigma$-finite measures. In a sense, this is even more natural, since for interesting Markov chains there are usually no invariant finite measures.\footnote{The necessity of considering $\sigma$-finite measures was indicated in the early paper \cite{KV}.} In the theory of AF-algebras, this problem is very important and related to representations
of type II$_{\infty}$ and to $\sigma$-finite traces. However, the theory of
 $\sigma$-finite Markov chains is not sufficiently developed.

\section{The geometry of projective limits of simplices}

\subsection{Projective limits of simplices and the equivalence of the languages}

An equivalent and rather geometric version of the theory of equipped graphs (a Bratteli diagram $\Gamma$ + a system of cotransition probabilities $\Lambda$) is the theory of projective limits of finite-dimensional simplices.

First we will show how, given a pair $(\Gamma,\Lambda)$, i.e., an equipped graph, one can canonically define a projective limit of finite-dimensional simplices. Then we will see that one can also pass in the opposite direction, from projective limits to equipped graphs.

Denote by $\Sigma_n$ the finite-dimensional simplex of formal convex combinations of vertices  $ v \in\Gamma_n$ of the $n$th level. It is natural to regard this simplex as the set of all probability measures on its vertex set
 $\Gamma_n$. We introduce affine projections
 $\pi_{n,n-1}: \Sigma_n \to \Sigma_{n-1}$; it suffices to define them for each vertex
$v \in \Gamma_n$. Obviously, these projections can be regarded as a system of cotransition probabilities
$\Lambda$, and the images of vertices $v$ are points of the previous simplex, i.e., probability vectors:
$$\pi_{n,n-1}(\delta_v)=\sum\limits_ {u: u\prec v} \lambda_v^u \delta_u; $$ this map is extended by linearity to the whole simplex $\Gamma_n$. The vertex $\emptyset$ corresponds to the zero-dimensional simplex consisting of a single point. Degeneracies are allowed (i.e., vertices may coalesce under the projection). Projections
 $\pi_{n,m}: \Sigma_n \rightarrow \Sigma_m$ of simplices with arbitrary numbers
 $m<n$, $m,n \in \mathbb N$, are defined as follows: $\pi_{n,m}=\prod_{i=m}^{n+1} \pi_{i,i-1}$. Having the data $\{\Sigma_n, \pi_{n,m}\}$, we can, on the one hand, define the projective limit of the simplices and, on the other hand, recover the corresponding graph (and then paths in this graph):  the vertices of $\Gamma_n$ coincide with the vertices of the simplex $\Sigma_n$, and the edges are found from the nonzero components of the vectors $\pi_{n,n-1}$, $n \in N$.

Denote by ${\mathcal M}=\prod_{n=0}^{\infty}\Sigma_n$ the direct product of the simplices $\Sigma_n$, $n\in \mathbb N$,
with the product topology.

\begin{definition}
The projective limit space of a sequence $\{\Sigma_n\}_n$ of simplices with respect to a system of projections
 $\{\pi_{n,m}\}$ is the following subset of the direct product ${\mathcal M}$:
  \begin{align*}
  \lim_{n\to\infty} (\Sigma_n, \pi_{n,m}) &\equiv\big \{\{x_n\}_n;\;\pi_{n,n-1}(x_n)=x_{n-1};\; n=1,2, \dots\big\}
  \\
  &\equiv (\Sigma_{\infty},\Lambda) \subset \prod_{n=0}^{\infty}\Sigma_n =\mathcal M.
  \end{align*}
\end{definition}

\begin{statement}
The projective limit space $\Sigma_{\infty}$ is always a nonempty, convex, closed, and hence compact subset in  $\mathcal M$, which is a (possibly, infinite-dimensional) Choquet simplex.
\end{statement}

The affine structure of the direct product $\mathcal M$ determines the affine structure of the limiting space; the fact that it is nonempty and closed is obvious. It remains only to check that every point of the limiting space has a unique decomposition over its Choquet boundary. This is done in the next subsection.

We differentiate between the projective limit space and the ``projective limit structure,'' meaning that it is important to consider not only the limiting space itself, i.e., an infinite-dimensional simplex,  but also the structure of prelimit simplices and their projections.

We will show that, given a projective limit of simplices, one can recover the corresponding graph, the path space, and the system of cotransition probabilities; and the projective limit constructed from this system according to the above rule coincides with the original one. This will establish a tautological relation between two languages: that of pairs \{a Bratteli diagram, a system of cotransition probabilities\} on the one hand, and that of projective limits of finite-dimensional simplices on the other hand.

Indeed, let $\{\Sigma_n\}$, $n\in \mathbb N$, be a projective limit of finite-dimensional simplices and $\{\pi_{n,m}\}$ be a coherent system of projections:
$$
 \pi_{n,m}:\Sigma_n \rightarrow \Sigma_m,\quad n\geq m,\ n,m\in \mathbb N.
$$
Take the vertices of
$\Sigma_n$ as the vertices of the $n$th level of $\Gamma$; a vertex $u$ of the $n$th level precedes a vertex $v$ of the $(n+1)$th level if the projection $\pi_{n+1,n}$ sends $v$ to a point of $\Sigma_n$
whose barycentric coordinate with respect to $u$ is positive. As a system of transition probabilities we take the system of vectors $\{\lambda_v^u\}$ related to the projections $\pi_{n+1,n}$ as described above.

In what follows, given a projective limit of simplices, we will use the  graph (of vertices of all simplices) canonically associated with this limit, its path space, etc., and,  similarly, speak about the projections of simplices canonically associated with an equipped graph.

\subsection{Properties of the limiting space, extremality and almost extremality}

Consider an arbitrary projective limit of finite-dimensional simplices
$$ \Sigma_1 \leftarrow \Sigma_2 \leftarrow \dots\leftarrow \Sigma_n\leftarrow \Sigma_{n+1}\leftarrow
\dots \leftarrow\Sigma_{\infty}\equiv \Sigma(\Gamma,\Lambda).$$
First of all, we define the limiting projections $\pi_{\infty,m}:\Sigma_{\infty}\rightarrow \Sigma_m$  for every $m$
as the limits $\lim_n \pi_{n,m}$: obviously, the images $\pi_{n,m}\Sigma_n$, regarded as subsets in $\Sigma_m$, decrease monotonically as $n$ grows, and their intersection is a set denoted by
$\Omega_m= \bigcap_{n:n>m} \pi_{n,m}\Sigma_n$; the sets $\Omega_m$ are convex closed subsets of the finite-dimensional simplices
$\Sigma_m$, $m=1,2, \dots$, and the limiting projections are epimorphic maps of the limiting space
$\Sigma_{\infty}$ onto these sets:
$$\pi_{\infty,m}:\Sigma_{\infty}\rightarrow \Omega_m. $$

It would be more economical to consider the projective limit $$\Omega_1\leftarrow \Omega_2 \leftarrow \dots \leftarrow\Omega_n \leftarrow \dots \leftarrow\Omega_{\infty}=\Sigma(\Gamma,\Lambda)$$
with the epimorphic projections $\pi_{n,m}$ restricted to $\Omega_n$ and, by definition, with the same limiting space. However, finding the sets $\Omega_n$ explicitly is an interesting problem equivalent to the main problem of finding all invariant measures.\footnote{In particular, an explicit form of the compact sets
$\Omega_n$ is known in very few cases, even among those ones where one knows the central measures. Even for the Pascal graph, they are interesting and rather complicated convex compact sets; and, for instance, in the case of the Young graph, the author does not know a description of these sets as clear as in the case of the Pascal graph.}

Every point of the limiting space, i.e., a sequence  $\{x_m\}$ with $x_m \in \Sigma_m$, $\pi_{m,m-1}x_m=x_{m-1}$,
defines, for every $m$, a sequence of measures $\{\nu_n^m\}_n$ on the simplex $\Sigma_m$, namely,
$\nu_n^m=\pi_{n,m}(\mu_n)$, where the measure $\mu_n$ is the (unique) decomposition of $x_n$ in terms of the extreme points of the simplex $\Sigma_n$. Of course, the barycenter of each of the measures
$\nu_n^m$ in $\Sigma_m$ is $x_m$, and this sequence of measures  is coarsening in a natural sense and weakly converges in $\Sigma_m$ as $n \to \infty$ to a measure $\nu_{x_m}$ concentrated on the subset
 $\Omega_m \subset \Sigma_m$. Obviously, in this way one can obtain all points of the limiting space, i.e., all measures with given cotransition probabilities.

 A point of an arbitrary convex compact space $K$ is called {\it extreme} if there is no nontrivial convex combination of points of $K$ representing this point; the set of extreme points is called the Choquet boundary of $K$ and denoted by
 $\ex K$. A point is called {\it almost extreme} if it lies in the closure
 $\overline{\ex}(K)$ of the Choquet boundary. Recall that an affine compact space in which every point has a unique decomposition into a convex combination of extreme points is called a Choquet simplex.

Now we give a general criterion of extremality and almost extremality for points of a projective limit of simplices.

\begin{statement}
{\rm1.} A point $\{x_n\}$ of a projective limit of simplices is extreme if and only if for every $m$ the weak limit of the measures $\nu_{x_m}$ on the simplex $\Sigma_m$ is the $\delta$-measure at $x_m$:
  $\lim_n \nu_n^m \equiv \nu_{x_m}= \delta_{x_m}$.

{\rm2.} A point $\{x_n\}$ is almost extreme if for every $m$ and every neighborhood
 $V(x_m)$ of $x_m \in \Sigma_m$ there exists an extreme point $\{y_n\}$ of the limiting space such that
 $y_m \in  V(x_m)$.

{\rm3.} Every point $\{x_n\}$ of a projective limit of simplices has a unique decomposition in terms of extreme points (Choquet decomposition), which is  defined via the measures $\nu_{x_m}$.
\end{statement}

\begin{corollary}
A projective limit of finite-dimensional simplices is a (in general, infinite-dimensional) Choquet simplex.
\end{corollary}

One can easily prove that the converse is also true: every separable Choquet simplex can be represented as a projective limit of finite-dimensional simplices, but, of course, such a representation is far from being unique. However, it is worth noting that the simplex of invariant measures for an action of a nonamenable group on a compact space is separable, though its possible approximation is not generated by finite approximations of the action; thus there arises a nontrajectory finite-dimensional approximation of the action, which, apparently, has not been studied.

\begin{remark}{\rm
Most probably, the first two claims of the proposition can be extended to projective limits of arbitrary convex compact spaces.}
\end{remark}

Recall that among separable Choquet simplices one singles out so-called {\it Poulsen simplices} for which the set of extreme points is dense; such a simplex is unique up to an affine isomorphism and universal in the class of all separable simplices. One can easily give an example of a Poulsen simplex in our terms.

\begin{statement}
Consider a projective limit of simplices $\Sigma_n$ satisfying the following property: for every $m$ the union
$$
\bigcup_{n;t}\big\{\pi_{n,m}(t);\; t\in \ex(\Sigma_n), \;n=m,m+1,\dots\big\}
$$
over all vertices of  $\Sigma_n$ and all $n>m$ is dense in
 $\Sigma_m$. Then the limiting space is a Poulsen simplex.
\end{statement}

Clearly, such a simplex can be constructed by induction, and the criterion obviously implies that the set of its extreme points is dense.

Simplices with closed Choquet boundary are called {\it Bauer simplices}. Between Bauer and Poulsen simplices, there are many intermediate types of simplices. In the literature on convex geometry and the theory of invariant measures, this subject has been repeatedly discussed. However, it seems that these and similar properties of infinite-dimensional simplices have never been considered in relation to projective limits and the theory of graded graphs and corresponding algebras. Each of these properties has an interesting interpretaion in the framework of these theories. The author believes that the following class of simplices (or even convex compact spaces) is useful for applications: {\it an almost Bauer simplex is a simplex whose Choquet boundary is open in its closure.}

\subsection{All boundaries in geometric terms}

The following definition is a paraphrase of the definition of Martin boundary in terms of projective limits.

\begin{definition} A point   $\{x_n\}\in \Sigma_n$ of a projective limit of simplices belongs to the Martin boundary if there is a sequence of vertices $\alpha_n \in \ex(\Sigma_n)$, $n =1,2, \dots$, such that for every $m$ and an arbitrary
neighborhood $V_{\epsilon}(x_m)\subset \Sigma_m$ there exists $N$
such that  $$\pi_{n,m}(\alpha_n) \in V_{\epsilon}(x_m)$$
 for all $n>N$. Less formally, a point of the limiting simplex belongs to the Martin boundary if there exists a sequence of vertices that weakly converges to this point (``from the outside'').
 \end{definition}

This sequence itself does not in general correspond to a point of the projective limit $\Sigma_{\infty}$,
but it is a point of the space $\mathcal M$ (the direct product of the simplices $\Sigma_n$), and it makes sense to say that its components approach the components of a point of the projective limit, which belongs to the Martin boundary by definition. The condition of belonging to the Martin boundary is a weakening of the almost extremality criterion, hence the following assertion is obvious.

\begin{statement}
The Martin boundary contains the closure of the Choquet boundary.
\end{statement}

However, there are examples where the Martin boundary contains the closure of the Choquet boundary as a proper subset.
Such an example, related to random walks, will be described in a joint paper by the author and A.~V.~Malyutin, which is now in preparation. A question arises: can one describe the Martin boundary in terms of the limiting simplex itself? In other words,
can one say what other points (except those lying in the closure of the Choquet boundary) belong to the Martin boundary? The author tends to believe that this cannot be done, since the answer to the latter question depends not only on the geometry of the limiting simplex itself, but also on how it is represented as a projective limit.

\subsection{The probabilistic interpretation of properties of projective limits}

Parallelism between considering pairs \{a graded graph, a system of cotransition probabilities\} on the one hand and considering projective limits of simplices on the other hand means that the latter subject has a probabilistic interpretation. It is useful to describe it without appealing to the language of pairs. Recall that in the context of projective limits a path is a sequence
$\{ t_n\}_n$ of vertices $t_n \in \ex\Sigma_n$ that agrees with the projections $\pi_{n,n-1}$ for all $n\in \mathbb N$ in the following sense: $\pi_{n,n-1}t_n$ has a nonzero barycentric coordinate with respect to $t_{n-1}$. First of all, every point $x_{\infty} \in \Sigma_{\infty}$ of the limiting simplex is a sequence  $\{x_n\}$ of points of the simplices $\Sigma_n$ that agrees with the projections: $\pi_{n,n-1}x_n=x_{n-1}$, $n \in \mathbb N$. As an element of the simplex, $x_n$ determines a measure on its vertices, and, since all these measures agree with the projections, $x_{\infty}$ determines a measure $\mu_x$
 on the path space with fixed cotransition probabilities. Conversely, every such measure comes from a point $x_\infty$. Thus the limiting simplex is the simplex of all measures on the path space with given cotransition probabilities. The extremality of a point $\mu\in \ex(\Sigma_{\infty})$ means the ergodicity of the measure $\mu$, i.e., the triviality with respect to $\mu$ of the tail $\sigma$-algebra on the path space. The above extremality criterion  has a simple geometric interpretation, on which we do not dwell.

So, we have considered the following boundaries of a projective limit of simplices (or an equipped graph):

  {\it  the Poisson--Furtsenberg boundary}  $\subset$ {\it the Dynkin boundary} = {\it the Choquet boundary} $\subset$ {\it the closure of the Choquet boundary} $\subset$
    {\it the Martin boundary}  $\subset$ {\it the limiting simplex}.

The first boundary is understood as a measure space; all inclusions are in general strict; the answer to the question of whether the Martin boundary is a geometric object (i.e., whether it can be defined in purely geometric terms, rather than via approximation) is most probably negative.

We summarize this section with the following conclusion:
{\it the theory of equipped graded graphs  {\rm(}i.e., pairs $\{$a graded graph $+$ a system of cotransition propabilities$\}${\rm)}
is identical to the theory of Choquet simplices regarded as projective limits of finite-dimensional simplices.}

\section{The intrinsic topology on a projective limit of simplices}

\subsection{The definition of the intrinsic topology on an inductive limit}

We proceed to our main goal, which is to construct an approximation of a projective limit of simplices, i.e., a simplex of measures with a given cocycle, and to define the ``intrinsic metric (topology)'' on this limit. This metric was defined in  \cite{V} on path spaces of graphs, only for central measures and under some additional conditions on the graph (the absence of vertices with the same predecessors). Here we give this definition in its natural generality, for an arbitrary graded graph and an arbitrary system of cotransition probabilities  (see Sec.~2), and, most importantly, we consider the whole limiting simplex and not only its Choquet boundary. This allows us to study the boundary for graphs with nonstandard (noncompact) intrinsic metrics. We formulate definitions and results both in terms of equipped graded graphs and in terms of projective limits of simplices spanned by the vertices of different levels.

We start with the definition of an important topological operation which will be repeatedly used, that of ``transferring a metric.''

Let $(X, \rho_X)$ be a metric space and $\phi: X \to Y$ be a (Borel-)measurable map from $X$ to a Borel space $Y$;
assume that the preimages of points $\phi^{-1}(y)$, $y \in \phi(X)\subset Y$, are endowed with Borel probability measures
 $\nu_y$ that depend on $y$ in a Borel-measurable way;  $\phi$ will be called an {\it equipped map}.

\begin{definition}
The result of transferring the metric $\rho_X$ on the space $X$ to the Borel space $Y$ along the equipped map $$\phi:X\rightarrow Y$$ is the metric $\rho_Y$ on  $Y$ defined by the formula $$\rho_Y (y_1,y_2)=k_{\rho_X}(\nu_{y_1},\nu_{y_2}),$$
where $k_{\rho}$ is the classical Kantorovich metric on Borel probability measures on  $(X,\rho_X)$.
\end{definition}

1. Consider an equipped graph $(\Gamma, \Lambda)$ and the corresponding projective limit of simplices
 $\Sigma_{\infty}(\Gamma)$. Define an arbitrary metric $\rho=\rho_1$  on the path space $T(\Gamma)$ that agrees with the Cantor topology on
$T(\Gamma)$; denote by $k_{\rho_1}$ the Kantorovich metric on the space $\Delta(\Gamma)$ of all Borel probability measures on  $T(\Gamma)$ constructed from the metric  $\rho_1$ (see the definition below).

2. Given an arbitrary path $v \equiv\{v_n\}$, consider the finite set of paths $v(u)=\{u,v_2, \dots\}$
whose coordinates coincide with the corresponding coordinates of $v$ starting from the second one, and assign each of these paths the measure $\lambda_{v_2}^u$. We have defined an equipped map
$\phi_1:T(\Gamma) \rightarrow \Delta(\Gamma)=\Delta_1$, which sends a path to the measure $\sum\limits_ {u:u\prec v_2}\lambda_{v_2}^u \delta_v(u)$. It is more convenient to regard it as a map from the simplex
 $\Delta(\Gamma)$ to itself, by identifying a path with the $\delta$-measure at it.

Transferring the metric $\rho_1$ along the equipped map $\phi_1$, we obtain a metric
 $\rho_2$ on a subset $\Delta_2=\phi(\Delta_1)$ of the simplex $\Delta (\equiv \Delta_1(\Gamma))$.

3. In a similar way we define the map $\phi_2$ that sends every measure from $\Delta_2$ concentrated on paths of the form
$\{u_1,v_2, \dots\}$, ${u_1\!\prec\!v_2}$, to the measure on the finite collection of paths of the form $\{u_1,u_2,v_3, \dots\}$
whose coordinates coincide with $v_i$  starting from the third one and the second coordinate $u_2$ runs over all vertices $u_2 \prec v_3$ with probabilities $\lambda_{v_3}^{u_2}$. Again transferring the metric
 $\rho_2$ from  the space $\Delta_2$ along the equipped map $\phi_2$, we obtain a metric $\rho_3$ on the image $\Delta_3 \equiv \phi_2(\Delta_2)=\phi_2\phi_1(\Delta)$.

Note that the images of the maps $\phi_n$, i.e., the sets $\Delta_n$, are simplices, but their vertices are no longer $\delta$-measures on the path space, but measures with finite supports of the form $\sum\limits_ {u_1,u_2, \dots, u_k} \lambda_{u_2}^{u_1}\cdots \lambda_{v_{k+1}}^{u_k}\cdot \delta_{u_1,\dots, u_k,v_{k+1}, \dots}$. The definition of the simplices $\Delta_n$ does not depend on the metrics $\rho_n$.

 4. Continuing this process indefinitely, we obtain an infinite sequence of metrics on the decreasing sequence of simplices
 \begin{align*}
 \Delta_n&=\phi_{n-1}(\Delta_{n-1})=\phi_n\phi_{n-1}\dots \phi_1(\Delta_1),
 \\
  \Delta&=\Delta_1 \supset \Delta_2 \supset \Delta_3 \dots, \quad  \bigcap_n \Delta_n =\Delta_{\infty}.
  \end{align*}
  Thus we have a sequence of equipped maps of the decreasing sequence of simplices
   $$\Delta_1 \to \Delta_2\to \dots\to \Delta_n \to \dots \to\Delta_{\infty}.$$

First we mention an assertion that does not involve the metric.

\begin{statement}
The intersection $\Delta_{\infty}$ of all simplices $\Delta_n$ consists exactly of those measures on the path space
 $T(\Gamma)$ (i.e., those points of the simplex  $\Delta(\Gamma)$ of all measures) that have given cotransition probabilities (given cocycle), and, therefore, this intersection coincides with the projective limit of the simplices:
$$\Delta_{\infty}=\Sigma_{\infty}(\Gamma).$$
\end{statement}

Of more importance is the following fact.

\begin{theorem}
There exists a limit $\lim_{n \to \infty}\rho_n = \rho_{\infty}$ of metrics on the space  $\Delta_{\infty}(=\Sigma_{\infty}(\Gamma))$. The limiting simplex $\Sigma_{\infty}(\Gamma)$ equipped with this metric is not in general compact, so that $\rho_{\infty}$ does not generate the projective limit topology.
\end{theorem}

\begin{proof}
We will give an explicit description of the limiting ``intrinsic'' metric, using more detailed information on the metrics
 $\rho_n$. To this end, we should remind the definition of the Kantorovich metric on measures and the notion of coupling, which is actually used in the definition of transferring metrics.

\begin{definition}
A coupling of two Borel probability measures $\mu_1,\mu_2$ defined on two (in general, different) Borel spaces
$X_1,X_2$ is an arbitrary Borel measure $\psi$ on the product  $X_1\times X_2$ whose projections to the factors $X_1,X_2$ coincide with $\mu_1,\mu_2$. The set of all couplings for $\mu_1,\mu_2$ will be denoted by
$\Psi(\mu_1,\mu_2)$.
(Other names for this notion are ``bistochastic measure,'' ``polymorphism,'' ``Young measure,'' ``correspondence,'' etc.)

The Kantorovich metric on the simplex of measures on a metric space $(X,\mu)$ is defined as follows:
$$k_\rho(\mu_1,\mu_2)=\inf\left\{ \int\limits_ {X\times X} \rho(x_1,x_2)\,d\psi(x_1,x_2): \psi \in \Psi(\mu_1,\mu_2)\right\}.$$
\end{definition}

Above we defined metrics (i.e., distances between measures) by  recursion on $n$, each time applying coupling. But one can do this consistently, combining all conditions on successive couplings together. In the infinite case, this gives at once a formula for the limiting metric.

Assume that the metric space $X$ in the previous definition is endowed with a sequence of equipped maps
$$X=X_1 \to X_2 (\subset X_1) \to X_3 (\subset X_2)\to \dots \to  X_n (\subset X_{n-1})$$
(here $n$ is finite or infinite; in the second case, the last space should be replaced by the intersection $\bigcap_n X_n \equiv
 X_{\infty}$) and we want to define the distance between measures on the last space  ($X_n$ or $X_{\infty}$). This is exactly our situation, where the spaces $X_n=\Delta_n$ are the simplices determined by the maps that replace an initial segment of a path by a measure distributed on initial segments. The formula remains the same as in the classical case, the difference being in what one means by a coupling:
     $$K_n(\mu_1, \mu_2)=\inf \left\{\int\limits_ {X\times X} \rho(x_1, x_2)\,d\psi(x_1,x_2): \psi_n \in \Psi_n (\mbox{or}\
     \in \Psi_{\infty})\right\}.$$
Here the coupling $\psi_n$ runs over the set $\Psi_n$ consisting of measures on the space $X\times X$
that not only have given projections but are such that the projection of $\psi_n$ to each component agrees with the structure of the sequence of projections of the space $X=X_1 \to X_2 \to \dots $ itself. In other words, for every $n$ the coupling
$\psi_n$ is a mixture of the couplings $\psi_{n-1}$: this strict constraint is the difference with the usual procedure. Thus the above formula correctly defines all metrics, including the limiting metric on the simplex
$\Delta_{\infty}=\Sigma_{\infty}(\Gamma)$.

Although the limiting intrinsic metric depends on the initial metric, nevertheless the formula shows also that the topology determined by the limiting metric is the same for all initial metrics that agree with the topology of the simplex.\end{proof}

\subsection{Standardness}

\begin{definition}\label{def6}
An equipped graph $(\Gamma, \Lambda)$, as well as a projective limit of simplices $\lim_n(\Sigma, \pi_{n,m})$,
are called standard if the limiting simplex of measures $\Sigma_{\infty}$ endowed with the intrinsic metric is compact. In this (and only this) case the projective limit topology coincides with the intrinsic topology.

A (nonequipped) graph will be called standard if the limiting simplex of central measures is compact in the intrinsic metric. The standardness or nonstandardness of an equipped graph depends in general on the system $\Lambda$.
\end{definition}

This definition generalizes the definition of a standard graph given in \cite{V}: a standard graph in the sense of \cite{V} is standard in the sense of the above definition. More exactly, if we restrict Definition~\ref{def6} to the spaces of paths of length $n$ regarded as sequences of vertices, then we obtain the definition from \cite{V}. One may say that the new definition is a linearization (extending to linear combinations) of the previous one. This can also be stated as follows: we consider (instead of vertices of a given level) measures on the set of paths leading to these vertices with given cotransition probabilities. Therefore, all metrics and their limit are defined on sets of measures (rather than sets of vertices), which provides  a natural generality for the definition removing the restrictions on the graph previously imposed. From a practical point of view, of course, it is more convenient to check the standardness by considering vertices (diagrams) if this is possible.

An example of a graph with a noncompact intrinsic metric is given in \cite{V}; we only mention that this is, for instance, the graph of unordered pairs related to the notion of tower of measures.

We state without proofs the main facts, which were partially reported in
  \cite{V}  under additional assumptions.

\smallskip
1. For a standard graph (projective limit of simplices), every ergodic measure on paths enjoys a concentration property: for every $\epsilon>0$, for all sufficiently large $n$, the $n$th level vertices lying on a set of paths of measure
$> 1-\epsilon$ are contained in a ball of radius at most $\epsilon>0$ in the intrinsic metric (this is also called the ``limit shape'' property). This allows one, in the case of an arbitrary standard equipped graph, to search for all ergodic measures among the limits along paths in the intrinsic metric (rather than among the weak limits, according to the ergodic method). In the nonstandard situation, the ergodic method cannot be strengthened in this way: the set of weak limits  in this case is in fact greater than the set of limits in the intrinsic metric.

\smallskip
2. The tail filtration on the path space of a standard graph with respect to every ergodic measure is standard in the metric sense. (For the definition of a standard filtration and the standardness criterion in the metric category, see
 \cite{Dec}.)

\smallskip
3. The most important fact, which reproduces the theorem on lacunary isomorphism  \cite{V3} in the topological situation is as follows.

\begin{theorem}[Lacunarization theorem]
For every equipped graph $(\Gamma=\bigcup_n \Gamma_n,\Lambda)$ (respectively, for every projective limit of simplices
 $\lim_n \{\Sigma_n, \{\pi_{n,m}\}_{n,m}\}$), one can choose a subsequence of positive integers $n_k$, $k=1,2,\dots $,
 such that the equipped multigraph $\Gamma' = \bigcup_k \Gamma_{n_k}$ obtained by removing all levels between
 $n_k$ and $n_{k+1}$, $k=1,2, \dots$, and preserving all paths connecting them (respectively, the projective limit $\lim_k\{\Sigma_{n_k}, \{\pi'_{k,s}\}_{k,s}\}$  with the lumped system of projections $\pi'_{k,s}$, where $$\pi'_{k,k+1}=\prod_{i=n_k}^{i=n_{k+1}-1} \pi_{i,i+1})$$ is standard.
\end{theorem}

This means that standardness is a property of the projective limit, and not of the limiting simplex: by changing (lumping) the approximation one can change the intrinsic topology and make it equivalent to the projective limit topology, even if they were distinct before lumping.

The interrelations between standardness and Bauerness of the limiting simplex need further study.

\end{document}